\RequirePackage{fix-cm}
\documentclass[smallcondensed,envcountsame]{svjour3} 
\smartqed  
\usepackage{graphicx, amssymb}
\usepackage{times,a4wide,mathrsfs}
\usepackage{amsmath}
\usepackage{amsfonts}

\newcommand{\C}{\mathbb{C}}
\newcommand{\ZZ}{\mathbb{Z}}
\newcommand{\RR}{\mathbb{R}}
\newcommand{\QQ}{\mathbb{Q}}

\newcommand{\PP}{\mathbb{P}}

\newcommand{\GG}{\mathbb{G}}
\newcommand{\OO}{\mathcal O}

\DeclareMathOperator{\pic}{Pic}
\DeclareMathOperator{\NS}{NS}

\DeclareMathOperator{\gr}{Gr}

\newcommand{\rom}{\romannumeral}
\DeclareMathOperator{\psef}{Psef}
\DeclareMathOperator{\hpsef}{Hpsef}
\DeclareMathOperator{\eff}{Eff}
\DeclareMathOperator{\nef}{Nef}

\DeclareMathOperator{\vol}{vol}
\DeclareMathOperator{\coeff}{coeff}

\newtheorem{convention}{Conventions}

\newtheorem{nonumbering}{Theorem}

\newtheorem{nonumberingc}{Corollary}

 \journalname{}

\begin{document}

\title{A note about connectedness theorems \`a la Barth}

\author{Robert Laterveer}

\institute{CNRS - IRMA, Universit\'e de Strasbourg \at
              7 rue Ren\'e Descartes \\
              67084 Strasbourg cedex\\
              France\\
              \email{laterv@math.unistra.fr}   }

\date{Received: date / Accepted: date}
%\date{}

\maketitle

\begin{abstract} We prove Barth--type connectedness results for low--codimension smooth subvarieties with good numerical properties inside certain ``easy'' ambient spaces (such as homogeneous varieties, or spherical varieties). The argument employs some basics from the theory of cones of cycle classes, in particular the notion of bigness of a cycle class.
 \end{abstract}



\section{Introduction}

The Mother Of All Connectedness Theorems is Barth's theorem. In its original version, Barth's theorem is about the cohomology of low--codimensional smooth subvarieties of projective space:

\begin{theorem}[Barth \cite{Ba}]\label{barth} Let $X\subset\PP^{n+r}(\C)$ be a smooth subvariety of dimension $n$. Then restriction induces isomorphisms
  \[ H^j(\PP^{n+r}(\C),\QQ)\ \xrightarrow{\cong}\ H^j(X,\QQ)\ \ \ \hbox{for\ all\ }j\le n-r\ .\]
  \end{theorem}
  
 Hartshorne \cite{Har} found a nice proof of theorem \ref{barth} based on the hard Lefschetz theorem. Subsequent extensions of Barth's theorem also establish connectedness results for homotopy groups, as well as for low--codimensional subvarieties of other ambient spaces, such as Grassmannians, rational homogeneous varieties or abelian varieties (cf. \cite{FL}, \cite{Fu}, \cite[Chapter 3]{Laz} for comprehensive overviews).
As is made clear by results of Debarre, in certain ambient spaces $P$ a connectedness result holds for any subvariety $X$ with an appropriate intersection--theoretic behaviour in $P$:

\begin{theorem}[Debarre \cite{Deb}]\label{OD} Let $P$ be a product of projective spaces or a Grassmannian, with $\dim P=n+r$. Let $X\subset P$ be a smooth subvariety of dimension $n\ge r+1$ which is {\em bulky\/} (i.e., $X$ meets all $r$--dimensional subvarieties of $P$). Then
  $X$ is simply connected.
  \end{theorem}
 
 Results similar in spirit have been obtained by Arrondo--Caravantes \cite{AC}, and by
 Perrin \cite{Per}, \cite{Per2}:
 
 \begin{theorem}[Arrondo--Caravantes \cite{AC}]\label{AC} Let $P$ be the Grassmannian of lines in a projective space, with $\dim P=n+r$. Let $X\subset P$ be a smooth bulky subvariety of dimension $n\ge r+2$. Then
   \[ \pic(X)=\ZZ\ .\]
   \end{theorem}

 \begin{theorem}[Perrin \cite{Per2}]\label{NP} Let $P$ be a rational homogeneous variety with Picard number $1$. Let $X\subset P$ be a smooth bulky subvariety of codimension $r$, and assume $2r\le \coeff(P)-2$ (here, $\coeff(P)$ is a number in between $0$ and $\dim P$, defined in \cite[Definition 0.9]{Per2}). Then the N\'eron--Severi group $NS(X)$ of $X$ has rank $1$:
  \[ \NS(X)= \ZZ\ .\]
  \end{theorem}

  In this note, we aim for similar connectedness results for subvarieties that have certain intersection--theoretic properties (such as bulkiness). Our main result is a cohomological version of theorem \ref{OD}. This result applies to any ambient space $P$ for which the cone $\eff^n(P)$ of effective codimension $n$ algebraic cycles modulo numerical equivalence is a closed cone (in particular, this applies when $P$ is a spherical variety, cf. corollary \ref{fmss}).
  
  \begin{nonumbering}[=theorem \ref{main}] Let $n,r$ be positive integers with $n\ge r+1$. Let $P$ be a smooth projective variety of dimension $n+r$, and assume
there is equality
  \[ \eff^n(P)=\psef^n(P)\ \]
  (i.e., the cone $\eff^n(P)$ is a closed cone).
   
 Let $X\subset P$ be a smooth closed subvariety of dimension $n$, and assume $X$ is strictly nef.
  Then the push--forward map
    \[  H^1(X,\QQ)_{}\ \to\ H^{2r+1}(P,\QQ) \]
    is injective.
\end{nonumbering}  
  
For the definition of ``strictly nef'', cf. definition \ref{sn}; on a homogeneous variety $P$, strict nefness is equivalent to bulkiness (remark \ref{compare}), which connects theorem \ref{main} to theorem \ref{OD}.
The proof of theorem \ref{main} is a very straightforward adaptation of Hartshorne's proof \cite{Har} of Barth's theorem using the hard Lefschetz theorem. The ampleness in Hartshorne's proof is replaced by ``bigness'' (in the sense of: being in the interior of the pseudo--effective cone of codimension $r$ cycles). Indeed, thanks to work of L. Fu \cite{LFu}, bigness of the class $[X]$ in the space $N^r(P)$ (of codimension $r$ cycles modulo numerical equivalence) is (under certain conditions) sufficient to obtain a connectedness result.

We establish some variants of theorem \ref{main} that similarly exploit this notion of bigness: in one variant (theorem \ref{main2}), there is no assumption on the ambient space $P$ but the assumptions on $X$ are stronger. As an application of theorem \ref{main2}, we obtain in particular the following improvement on the above--cited result of Perrin:
 
 \begin{nonumberingc}[=corollary \ref{NP+}] Let $X$ and $P$ be as in theorem \ref{NP}. Then
   \[ \pic(X)=\ZZ\ .\]
 \end{nonumberingc}
 
 In another variant result (proposition \ref{spher}), we show that when $P$ is a spherical variety, there is still a certain connectedness even for subvarieties $X$ that may fail to be bulky. 
 
 Finally, we include a conditional result (theorem \ref{main3}) that proves connectedness for cohomology of degree $>1$. This result is conditional, because (apart from the codimension $2$ case) we need to assume the standard Lefschetz conjecture $B(X)$ for the subvariety $X$. Theorem \ref{main3} implies in particular a conditional improvement on the above--cited result of 
 Arrondo--Caravantes:
 
 \begin{nonumberingc}[=corollary \ref{AC+}] Let $X$ and $P$ be as in theorem \ref{AC}, and suppose either $r=2$ or $B(X)$ holds. Then
   \[ H^2(X,\ZZ)=\ZZ\ .\]
   \end{nonumberingc}
  
We present two more applications of a similar ilk (corollaries \ref{AC++} and \ref{NP++}). Just like corollary \ref{AC+}, these applications prove a certain connectedness result for bulky subvarieties of codimension $2$ and for bulky subvarieties verifying the standard Lefschetz conjecture.

\vskip0.5cm
\begin{convention} All varieties will be irreducible projective varieties over $\C$. A subvariety will always be a {\em closed\/} subvariety.

\end{convention}

%\section{Preliminary}

\section{Cones of cycle classes}

\begin{definition} Let $M$ be a smooth projective variety of dimension $m$. Let $N^j(M)$ denote the $\RR$--vector space of codimension $j$ algebraic cycles on $M$ (with $\RR$--coefficients) modulo numerical equivalence. Let
  \[ \eff^j(M)\ \ \subset\ N^j(M)\]
  be the cone generated by effective algebraic cycles. Let
  \[ \psef^j(M):=\overline{\eff^j(M)}\ \ \subset\ N^j(M) \]
  be the closure of the cone generated by effective algebraic cycles. A class $\gamma\in N^j(M)$ is called {\em big\/} if $\gamma$ is in the relative interior of $\psef^j(M)$.
  
  The intersection product defines a perfect pairing 
  \[ N^j(M)\times N^{m-j}(M)\ \to\ N^m(M)\cong \RR\ .\]
  Let
  \[ \nef^j(M)\ \ \subset\ N^j(M) \]
  be the cone dual to $\psef^{m-j}(M)$ under this pairing.
   \end{definition}

  The pseudo--effective cone $\psef^j(M)$ is studied for instance in \cite{DV}, \cite{Le}, \cite{FLe}, \cite{FLe2}, \cite{QLi}.
  There is another notion of bigness, which is a priori more stringent:
  
  \begin{definition} Let $M$ be a smooth projective variety. Let $N^\ast$ denote the coniveau filtration on cohomology \cite{BO}. Let
  \[ \hpsef^j(M)\ \ \subset N^j H^{2j}(M,\RR) \]
  be the closure of the cone generated by effective algebraic cycles. A class $\gamma\in N^j H^{2j}(M,\RR)$ is called {\em homologically big\/} if $\gamma$ is in the relative interior of $\hpsef^j(M)$.
  \end{definition}
  
  \begin{remark}\label{D} The ``homologically pseudo--effective cone'' $\hpsef^j(M)$ is considered for instance in \cite{Voi} and \cite{LFu}. If Grothendieck's standard conjecture $D(M)$ is true (i.e., homological and numerical equivalence coincide on $M$), then there is a natural isomorphism
    \[ N^j H^{2j}(M,\RR)\cong N^j(M)\ ,\]
    and so the two notions of bigness coincide. In particular, since we know the standard conjecture $D$ is true in codimension $1$ and $2$ \cite[Corollary 1]{Lie} and for curves (\cite[Corollary 1]{Lie}, or alternatively \cite[Proposition 1.1]{CTS}), the two notions of bigness coincide for $j=1$, for $j=2$ and for $j=n-1$. In general, in the absence of $D(M)$, we only know that a homologically big class in $N^j H^{2j}(M,\RR)$ projects to a big class in $N^j(M)$. For more on the standard conjectures, cf. \cite{K0}, \cite{K}.
    \end{remark}

Thanks to work of Lehmann, there exists a nice volume--type function for cycle classes. This volume--type function acts as a bigness detector:

\begin{theorem}[Lehmann \cite{Le}]\label{vol} Let $X$ be a smooth projective variety of dimension $n$. Consider the homogeneous function defined as
  \[ \begin{split}  \widehat{\vol}\colon\ \ N^j(X)\ &\to\ \RR_{\ge 0}\ ,\\
                             \widehat{\vol}(\alpha)&:=\sup_{\phi, A} \{ A^n\}\ ,\\
    \end{split} \]
  where $\phi\colon Y\to X$ varies over all birational models of $X$, and $A$ varies over all big and nef $\RR$--Cartier divisors on $Y$ such that $\phi_\ast(A^j)-\alpha\in\psef^j(X)$. 
 This function has the property that $\widehat{\vol}(\alpha)>0$ if and only if $\alpha$ is big.
  \end{theorem}     
          
\begin{proof} This is \cite[Section 7]{Le}.
\end{proof}

%\subsection{Motives}

%\begin{definition} Let $\MM_{\rm hom}$ and $\MM_{\rm num}$ denote the (contravariant) category of pure motives with respect to homological equivalence (resp. numerical equivalence), as in \cite{Sch}, \cite{MNP}.

%Let $\MM_{\rm hom}^{ab}$ denote the full subcategory generated by motives of curves.
%\end{definition}

%\begin{proposition} Let $X$ be a smooth projective variety of dimension $n$. There exist motives
 % \[ h^1(X)\ ,\ h^{2n-1}(X)\ \in \MM_{\rm hom}\ ,\]
 % with the property that $H^\ast(h^1(X))=H^1(X,\QQ)$ and $H^\ast(h^{2n-1}(X))=H^{2n-1}(X,\QQ)$.
 % Moreover, we have
 % \[   h^1(X)\ ,\ h^{2n-1}(X)\ \in \MM^{ab}_{\rm hom}\ .\]
%\end{proposition}

\section{Strictly nef subvarieties}

In this section, we prove the main result of this note (theorem \ref{main}), which is about degree $1$ cohomology of smooth strictly nef subvarieties.

\begin{definition} Let $P$ be a smooth projective variety, and let $X\subset P$ be a closed irreducible subvariety of codimension $r$. We say that $X$ is {\em bulky\/} if $X$ meets every
dimension $r$ subvariety of $P$, i.e. 
 for every closed $r$--dimensional subvariety $a\subset P$, we have
   \[ X\cap a\not=\emptyset\ \]
   (here $\cap$ indicates set--theoretic intersection).
   \end{definition}

\begin{definition}\label{sn} Let $P$ be a smooth projective variety, and let $X\subset P$ be a closed irreducible subvariety of codimension $r$. We say that $X$ is {\em strictly nef\/} if
for every non--zero $a\in\eff_r(P)$ we have
  \[ [X]\cdot a>0\ \ \hbox{in}\ H_0(P,\RR)\cong\RR\ .\]
 \end{definition}
 
 \begin{remark} The definition of bulkiness seems to originate with \cite{Deb} (where it is called ``une sous--vari\'et\'e encombrante''). In \cite{Per2}, the adjective ``cumbersome'' is used instead of bulky.
 \end{remark}
 
 \begin{remark} Strictly nef divisors are studied in \cite{CCP}.
 \end{remark}
 
 \begin{remark}\label{compare} Any strictly nef subvariety is bulky. On a homogeneous variety $P$, the converse is true (indeed, any non--zero effective class on $P$ is represented by an effective cycle in general position with respect to $X$). On arbitrary varieties $P$, the converse is {\em not\/} true. (Here is an example that was kindly pointed out by the referee: Let $P_1,\ldots,P_{10}$ be 10 very general points on an elliptic curve $E\subset\PP^2$. Let $S\to\PP^2$ denote the blow--up with center the $10$ points $P_i$, and let $\bar{E}\subset S$ be the strict transform of $E$. One can check that $\bar{E}\subset S$ is bulky. On the other hand, the self--intersection $\bar{E}^2$ is negative, so $\bar{E}$ is not nef.)
 
 To recap, one could say that the notion of strict nefness (which is equivalent to bulkiness on homogeneous varieties) is the more natural notion for arbitrary varieties.
  \end{remark}

 \begin{example} Let $P$ be a homogeneous variety, and $X\subset P$ a smooth subvariety with ample normal bundle. Then $X$ is bulky \cite[Example 8.4.6]{Laz}. In particular, if $P$ is a simple abelian variety, every smooth subvariety $X\subset P$ is bulky \cite[Corollary 6.3.11]{Laz}.
 \end{example}

 \begin{definition}  Let $P$ be a smooth projective variety, and let $X\subset P$ be a closed irreducible subvariety of codimension $r$. 
 We will write
   \[  H^j(X)_{\text{van}} := \ker\bigl( H^j(X,\C)\ \to\ H^{j+2r}(P,\C)\bigr)\ .\]
   \end{definition} 
   
   \begin{remark} It follows from mixed Hodge theory that the kernel
     \[\ker\bigl( H^j(X,\QQ)\ \to\ H^{j+2r}(P,\QQ)\bigr) \] 
     is a Hodge sub--structure \cite{PS}. Thus, it makes sense to write $\gr^i_F H^j(X)_{\text{van}}$ (where $F^\ast$ denotes the Hodge filtration).
   \end{remark}

\begin{theorem}\label{main} Let $n$ and $r$ be positive integers with $n\ge r+1$. Let $P$ be a smooth projective variety of dimension $n+r$, and assume
there is equality
  \[ \eff^n(P)=\psef^n(P)\ \ \ \subset N^n(P)\]
  (i.e., the cone $\eff^n(P)$ is a closed cone).
   
 Let $X\subset P$ be a smooth subvariety of dimension $n$ which is strictly nef.
  Then
    \[  H^1(X)_{\text{van}}=0\ . \]
    \end{theorem}
    
  \begin{proof} Suppose $n>r+1$. There is a fibre diagram
    \[ \begin{array}[c]{ccc}
               X^\prime &\xrightarrow{\tau^\prime}& P^\prime\\
               \downarrow&&\ \ \downarrow{f}\\
               X&\xrightarrow{\tau}&\ \ P\ ,\\
               \end{array}\]
               where $P^\prime\subset P$ is a smooth complete intersection of dimension $n^\prime+r$, and $X^\prime\subset X$ is smooth of dimension $n^\prime$, and we have equality
               $n^\prime=r+1$.

   \begin{lemma} The class
     \[ (\tau^\prime)^\ast[X^\prime]\ \in N^r(X^\prime) \]
     is homologically big.
     \end{lemma}
     
     \begin{proof} First, since $r=n^\prime-1$ (i.e., $(\tau^\prime)^\ast [X^\prime]$ is a curve class), the notions of bigness and homological bigness are the same (remark \ref{D}). We are thus reduced to proving bigness, i.e. we need to prove $(\tau^\prime)^\ast[X^\prime]$ is in the relative interior of $\eff^r(X^\prime)$. Let $A\subset P$ be a codimension $r$ intersection of ample divisors. Then
     \[ A^\prime:=(\tau^\prime) f^\ast (A)\ \ \in N^r(X^\prime) \]
     is the class of a codimension $r$ intersection of ample divisors; as such, $A^\prime$ is in the relative interior of $\eff^r(X^\prime)$ (\cite[Lemma 2.11]{FL}, or alternatively theorem \ref{vol}).
     Hence, to prove bigness of $(\tau^\prime)^\ast[X^\prime]$, it suffices to prove that
     \begin{equation}\label{int}  (\tau^\prime)^\ast[X^\prime] -\epsilon A^\prime\ \ \in \psef^r(X^\prime)\ ,\end{equation}
     for some $\epsilon >0$ sufficiently small.
     
      Now let $D\in\nef^1(X^\prime)$. Then we have
     \[  \begin{split}  \Bigl((\tau^\prime)^\ast [X^\prime] -\epsilon A^\prime \Bigr) \cdot D&= \Bigl( (\tau^\prime)^\ast f^\ast ([X]-\epsilon A)\Bigr)  \cdot D\\
                                                                                       &= \bigl([X]-\epsilon A\bigr) \cdot f_\ast (\tau^\prime)_\ast (D)\\
                                                                                       &\ge 0\ ,\\
                                                                                       \end{split} \]
                                                                                       for some $\epsilon >0$ sufficiently small.
                                          Here, the first equality is just the fact that $f^\ast[X]=[X^\prime]$, and the second equality is the projection formula. As for the last line, note that $X\subset P$ is strictly nef, which combined with the assumption that  $\eff^n(P)$ is a closed cone implies that $[X]$ is strictly positive on $\psef^n(P)\setminus \{0\}$, i.e. $[X]$ is in the relative interior of $\nef^r(P)$.
    On the other hand, $\nef^1(X^\prime)\subset\psef^1(X^\prime)$, and so the push--forward $f_\ast(\tau^\prime)_\ast (D)$ is pseudo--effective, hence (by assumption) effective. This means that there exists $\epsilon>0$ such that $\bigl([X]-\epsilon A\bigr) \cdot f_\ast (\tau^\prime)_\ast (D)\ge 0$.  
  This proves the inclusion (\ref{int}), and hence the lemma.  
      \end{proof}
    
  Homological bigness is relevant to us, because of the following hard Lefschetz type result:
  
     \begin{lemma}[L. Fu \cite{LFu}]\label{lf} Let $M$ be a smooth projective variety of dimension $n$, and let $\gamma\in N^r H^{2r}(M,\QQ)$ be homologically big. Then the homomorphism ``cup product with $\gamma$'' induces an injection
    \[  \cup \gamma\colon\ \ \gr^0_F H^{n-r}(M,\C)\ \to\ \gr^r_F H^{n+r}(M,\C)\ \]
    (here $F^\ast$ denotes the Hodge filtration).
    \end{lemma}
    
  \begin{proof} This is \cite[Lemma 3.3]{LFu}. The proof exploits the second Hodge--Riemann bilinear relation, and is inspired by ideas of \cite{Voi}.
  \end{proof}  
 
   Applying lemma \ref{lf} to the homologically big class $(\tau^\prime)^\ast[X^\prime]\ \in N^r H^{2r}(X^\prime,\QQ)$, we find that
   \[   \cup (\tau^\prime)^\ast[X^\prime]\colon\ \  \gr^0_F  H^1(X^\prime,\C)\ \to\ \gr^{n^\prime-1}_F H^{2n^\prime-1}(X^\prime,\C)    \]
   is injective (and hence, for dimension reasons, an isomorphism). Using the fact that $\gr^1_F H^1$ is the complex conjugate of $\gr^0_F H^1$,
   we find that
   \[  \cup (\tau^\prime)^\ast[X^\prime]\colon\ \    H^1(X^\prime,\C)\ \to\  H^{2n^\prime-1}(X^\prime,\C)    \]  
   is also injective.
     On the other hand, it follows from the normal bundle formula that there is a factorization  
     \[\cup (\tau^\prime)^\ast[X^\prime]\colon\ \    H^1(X^\prime,\C)\ \xrightarrow{(\tau^\prime)_\ast}\  H^{2r+1}(P^\prime,\C)\ 
      \xrightarrow{(\tau^\prime)^\ast}\ 
          H^{2n^\prime-1}(X^\prime,\C)  \ .\]
         We can thus conclude that
         \[ (\tau^\prime)_\ast\colon\ \ H^1(X^\prime,\C)\ \to\ H^{2r+1}(P^\prime,\C) \]
         is injective.
         We have a commutative diagram
         \[ \begin{array}[c]{ccc}
          H^1(X,\C)&\xrightarrow{\tau_\ast}& H^{2r+1}(P,\C)\\
          \downarrow&&\downarrow\\
          H^1(X^\prime,\C)&\xrightarrow{(\tau^\prime)_\ast}& H^{2r+1}(P^\prime,\C)\\
          \end{array}\]
          % (NB: this can be seen by using cohomology with supports, and forgetting the support !!!!)
          where vertical arrows are injective (weak Lefschetz, note that $\dim P^\prime=2r+1$). It follows that
          \[ \tau_\ast\colon\ \ H^1(X,\C)\ \to\ H^{2r+1}(P,\C)\]
          is injective.

 %  The class $[X]\in N^r(P)$ is big, i.e. it lies in the interior of the pseudo--effective cone. This implies that we can find a positive--dimensional closed rational polyhedron $\Delta$ contained in the interior of $\psef^r(P)$ and containing the ray generated by $[X]$:
%   \[  \RR[X]\in\inte(\Delta)\ ,\ \ \ \Delta\subset\inte\bigl(\psef^r(P)\bigr)\ \ \ \subset N^r(P)\ .\]            
%  Moreover, since $P^\prime$ was obtained by taking generic hyperplane sections, we may suppose $P^\prime\subset P$ is in general position with respect to the finitely many generators of $\Delta$. Thus, denoting by $f$ the inclusion morphism $f\colon P^\prime\to P$, we have
%    \[ f^\ast(\Delta)\subset \eff^r(P^\prime)\ \ \ \subset N^r(P^\prime)\ .\]
%  Note that by weak Lefschetz (as $\dim P^\prime=n^\prime+r=2r+1$), we have an isomorphism 
 %   \[ f^\ast\colon\ \ H^{2r}(P,\RR)   \ \xrightarrow{\cong}\ H^{2r}(P^\prime,\RR)\ .\]     
 %  By assumption, $H^{2r}(P,\RR)$ is isomorphic to $N^r(P)$, and so we obtain a surjection
  %   \[ f^\ast\colon\ \ N^r(P)\ \to\ N^r(P^\prime)\ .\]
  %   This means that the class $[X]$ (which lies, as we have seen, in the interior of the polyhedron $\Delta$) projects (under $f^\ast$) to a class in the interior of $f^\ast(Delta)\subset\eff^r(P^\prime)$.
  %  But $f^\ast [X]=[X^\prime]\in N^r(P^\prime)$, and so $[X^\prime]$ is in the interior of $\eff^r(P^\prime)$, i.e. $[X^\prime]$ is big.            

  \end{proof}  
   
  As a corollary, we obtain the following:
  
  \begin{corollary}\label{fmss} Let $P$ be a smooth projective variety of dimension $n+r$, and suppose a connected solvable linear algebraic group acts on $P$ with finitely many orbits.
  Let $X\subset P$ be a smooth subvariety of dimension $n\ge r+1$ which is strictly nef. Then
    \[ H^1(X,\QQ)=0\ .\]
    \end{corollary}
    
   \begin{proof} For $P$ as in corollary \ref{fmss}, it is known that all cones $\eff^r(P)$ are closed rational polyhedral cones, generated by the orbit 
   closures \cite[Corollary to Theorem 1]{FMSS}. Theorem \ref{main} thus applies; this gives
     \[ H^1(X)_{\text{van}}=0\ .\]
    But $P$ has no odd cohomology since the cycle class map is an isomorphism \cite[Corollary to Theorem 2]{FMSS}, and so $H^1(X,\QQ)=0$.
    \end{proof}
    
    \begin{remark} Suppose $P$ is a Grassmannian or a product of projective spaces (of dimension $n+r$), and $X\subset P$ smooth and bulky (of dimension $n\ge r+1$) as in corollary \ref{fmss}. Then, as noted in the introduction, Debarre has proven that $X$ is simply connected \cite{Deb}. Can one also prove simple--connectedness in the more general set--up of corollary \ref{fmss} ?
    \end{remark}

Here is a variant of theorem \ref{main} where we make no assumption on the ambient space $P$.

\begin{theorem}\label{main2} Let $n,r$ be positive integers with $n\ge r+1$. Let $P$ be a smooth projective variety of dimension $n+r$. Let $X\subset P$ be a smooth subvariety of dimension $n$ that is strictly nef. Assume that
$\dim N^1(X)=1$.
  Then
    \[  H^1(X)_{\text{van}}=0\ . \]
    \end{theorem}

\begin{proof} This is similar to theorem \ref{main}. Again, in case $n>r+1$, we consider a fibre diagram
  \[ \begin{array}[c]{ccc}
               X^\prime &\xrightarrow{\tau^\prime}& P^\prime\\
               \downarrow&&\ \ \downarrow{f}\\
               X&\xrightarrow{\tau}&\ \ P\ ,\\
               \end{array}\]
               where $P^\prime\subset P$ is a generic smooth complete intersection of dimension $n^\prime+r$, and $X^\prime\subset X$ is smooth of dimension $n^\prime$, and we have equality
               $n^\prime=r+1$.
         Taking $P^\prime$ sufficiently generic, we will have $\dim N^1(X^\prime)=1$ (this follows from weak Lefschetz in case $n^\prime\ge 3$, and from Noether--Lefschetz in case $n^\prime=2$). Hence, to test the bigness of the curve class $(\tau^\prime)^\ast[X^\prime]$, it suffices to intersect with one ample divisor $D\in\nef^1(X^\prime)$. But any ample divisor is effective, and so the push--forward $f_\ast(\tau^\prime)_\ast(D)$ is effective. It follows that the intersection is positive, by strict nefness of $X$:
         \[     \begin{split}  (\tau^\prime)^\ast [X^\prime] \cdot D&= (\tau^\prime)^\ast f^\ast [X] \cdot D\\
                                                                                       &= [X]\cdot f_\ast (\tau^\prime)_\ast (D)\\
                                                                                       &>0\ .\\
                                                                                       \end{split} \]   
                                    We conclude that $(\tau^\prime)^\ast[X^\prime]$ is big. The rest of the argument is the same as theorem \ref{main}.                                                   
                               \end{proof}

Thanks to theorem \ref{main2}, we can ``complete'' certain results of Perrin:

\begin{corollary}\label{NP+} Let $P$ be a rational homogeneous variety with Picard number $1$, and $\dim P=n+r$. Let $X\subset P$ be a smooth bulky subvariety of dimension $n$, and assume $2r\le \coeff(P)-2$ (here, $\coeff(P)$ is a number in between $0$ and $\dim P$, defined in \cite[Definition 0.9]{Per2}). Then
  \[ \pic(X)= \ZZ\ .\]
  \end{corollary}

\begin{proof} Note that bulkiness and strict nefness coincide on $P$ (remark \ref{compare}). Perrin has proven \cite[Theorem 0.10]{Per2} that the N\'eron--Severi group $\NS(X)$ is $\ZZ$, so that $N^1(X)=\RR$. The result now follows from theorem \ref{main2}, in view of the exact sequence (coming from the exponential sequence)
  \[   H^1(X,\ZZ)\ \to\ H^1(X,\OO)\ \to\ \pic (X)\ \to\ \NS(X)\ \to\ 0\ .\]
  
\end{proof}

\section{Not so bulky subvarieties}

%\begin{definition} Let $P$ be a smooth projective variety, and let $X\subset P$ be a closed irreducible subvariety of codimension $r$. We say that $X$ is {\em slightly bulky\/} if for every $a\in\nef_r(P)$ we have
%  \[ [X]\cdot a>0\ \ \hbox{in}\ H_0(P,\QQ)\ .\]
% \end{definition}

In this section, we consider a refinement of theorem \ref{main} for certain special ambient spaces $P$. The connectedness result of this section (proposition \ref{spher}) improves on theorem \ref{main} because it applies to subvarieties $X$ that may fail to be bulky (cf. remark \ref{strict}).

\begin{definition}\label{sph} Let $G$ be a connected reductive algebraic group. A {\em spherical variety\/} is a normal $G$--variety for which there is a Borel subgroup $B\subset G$
with a dense orbit.
\end{definition}
 
 \begin{remark} More on spherical varieties can be found in \cite{Per3}, \cite{Br}, \cite{Br2} and the references given there.
 \end{remark}

  \begin{proposition}\label{spher} Let $P$ be a smooth projective spherical variety of dimension $n+r$. Let $X\subset P$ be a smooth subvariety of dimension $n\ge r+1$, verifying the following:
  
  \noindent
  (\rom1) $X$ is in general position with respect to the $n$--dimensional orbit closures on $P$;

 \noindent
 (\rom2) $X\subset P$ is big.
 
 Then
    \[  H^1(X,\QQ)=0\ .\]
  \end{proposition}

  \begin{proof} As before, in case $n>r+1$, we consider a fibre diagram
  \[ \begin{array}[c]{ccc}
               X^\prime &\xrightarrow{\tau^\prime}& P^\prime\\
             \ \   \downarrow{g}&& \downarrow\\
               X&\xrightarrow{\tau}&\ \ P\ ,\\
               \end{array}\]
               where $P^\prime\subset P$ is a smooth complete intersection of dimension $n^\prime+r$, and $X^\prime\subset X$ is smooth of dimension $n^\prime$, and we have equality
               $n^\prime=r+1$.

  \begin{lemma}\label{OK} The class $\tau^\ast[X]\in N^r(X)$ is big.
  \end{lemma}
  
  \begin{proof}
  As the cone $\eff^{r}(P)$ is generated by the $n$--dimensional orbit closures \cite{FMSS}, assumption (\rom1)
  implies that 
    \[ \tau^\ast \bigl(\eff^{r}(P)\bigr)\ \subset\ \eff^{r}(X)\ .\]
    Dually, this amounts to an inclusion
    \[  \tau_\ast \bigl( \nef^{n-r}(X)\bigr)\ \subset\ \nef^n(P)\ .\]
    Let $A\in N^1(P)$ denote the class of an ample divisor. The class $\tau^\ast(A^r)$ lies in the relative interior of $\eff^r(X)$. Hence, proving lemma \ref{OK} is equivalent to showing
    \begin{equation}\label{biginc} \tau^\ast[X] - \epsilon \tau^\ast (A^r)\ \ \in \eff^r(X)\ ,\end{equation}
    for some $\epsilon >0$ sufficiently small.
    
        Let $D\in\nef^{n-r}(X)$. As we have seen, $\tau_\ast(D)\in \nef^n(P)$. It follows that
      \[   \bigl(\tau^\ast[X] -\epsilon \tau^\ast(A^r)\bigr)\cdot D= \bigl([X]-\epsilon A^r\bigr)\cdot \tau_\ast(D) \ge 0\ ,\]
      for some $\epsilon >0$ sufficiently small. This proves inclusion (\ref{biginc}), and hence lemma \ref{OK}.   
      \end{proof}   
    
  \begin{lemma}\label{Xprime} The class $(\tau^\prime)^\ast[X^\prime]\in N^r(X^\prime)$ is homologically big.
  \end{lemma}
  
  \begin{proof} Since $\tau^\ast[X]$ is big (lemma \ref{OK}), we can write
    \[ \tau^\ast[X]= A^r + e\ \ \ \hbox{in}\ N^r(X)\ ,\]
    where $A$ is an ample divisor on $X$, and $e$ is an effective class. (Here, we have again used the fact that complete intersection classes $A^r$ are big; this is \cite[Lemma 2.11]{FLe}, or, alternatively, can be seen using the volume--type function of theorem \ref{vol}.) For a generic choice of $X^\prime$, the restriction $e^\prime=g^\ast(e)$ is still effective, and (obviously) $A^\prime=g^\ast(A)$ is still ample. It follows that
    \[ (\tau^\prime)^\ast[X^\prime] = (A^\prime)^r + e^\prime\ \ \ \hbox{in}\ N^r(X^\prime)\ \]
  is big. 
  
  Because $r=n^\prime-1$ (i.e., we look at a curve class on $X^\prime$) the class $(\tau^\prime)^\ast[X^\prime]$ is also homologically big (remark \ref{D}).  
  \end{proof}  
      
The rest of the argument is identical to that of theorem \ref{main}: Applying lemma \ref{lf} to the homologically big class $(\tau^\prime)^\ast[X^\prime]$, we find that
  \[ (\tau^\prime)_\ast\colon\ \ H^1(X^\prime,\QQ)\ \to\ H^{2r+1}(P^\prime,\QQ) \]
  is injective. The commutative diagram
    \[ \begin{array}[c]{ccc}
          H^1(X,\C)&\xrightarrow{\tau_\ast}& H^{2r+1}(P,\C)\\
          \downarrow&&\downarrow\\
          H^1(X^\prime,\C)&\xrightarrow{(\tau^\prime)_\ast}& H^{2r+1}(P^\prime,\C)\\
          \end{array}\]
       (where vertical arrows are injective by weak Lefschetz) then proves the proposition.      
    \end{proof}

  \begin{remark}\label{strict} Let $X$ be a smooth projective spherical variety. It is known \cite[Theorem 1.1]{QLi} that there are inclusions of cones
    \[ \nef^j(P)\ \subset\ \eff^j(P)\ \ \ \hbox{for\ all\ }j\ .\]
  That is, any bulky subvariety $X\subset P$ verifies hypothesis (\rom2) of proposition \ref{spher}. 
  
  We can say more: as shown in \cite{QLi}, there are ``many'' spherical varieties $P$ for which there are {\em strict\/} inclusions
    \[ \nef^j(P)\ \subsetneqq\ \eff^j(P)\ \ \ \hbox{for\ all\ }j\ . \]
 (More precisely: let $P$ be either a toric variety different from a product of projective spaces, or a toroidal spherical variety different from a rational homogeneous space. Then these inclusions are strict for all $j$ \cite[Theorem 1.2]{QLi}.) The conclusion is that in these cases proposition \ref{spher} gives a connectedness result even for subvarieties $X$ that fail to be bulky; it suffices that $X$ be only ``slightly bulky'', in the sense of hypothesis (\rom2).
   \end{remark}

%\begin{corollary}\label{AC+} Let $P$ be a Grassmannian of lines, with $\dim P=n+r$. Let $X\subset P$ be a smooth bulky subvariety of dimension $n$. Then
%  \[ H^2(X,\ZZ)=\ZZ\ .\]
 % \end{corollary}
  
 % \begin{proof} As mentioned in the introduction, Arrondo and Caravantes have proven \cite{AC} that $\pic(X)=\ZZ$. It follows from corollary \ref{spherhomo} that $H^2(X,{\mathcal O}%_X)=0$. The result now follows from the exponential sequence.
%  \end{proof}

\section{A conditional result} 

In this final section, we prove a conditional connectedness result for cohomology groups in degree $>1$. This result is conditional to one of the standard conjectures. The reason we need to assume a standard conjecture is that there might a priori be a difference between the two notions of bigness defined in section 2 (cf. remark \ref{D}).

\begin{theorem}\label{main3} Let $P$ be a smooth projective variety of dimension $n+r$, and $\tau\colon X\subset P$ a smooth subvariety of dimension $n$.
 Assume the following:
 
\noindent
(\rom1) There is an inclusion of cones
  \[    \nef^n(P)\ \subset\ \eff^n(P)\ ;\]
  
 \noindent
 (\rom2) $X\subset P$ is strictly nef;
 
 \noindent
 (\rom3)  There is an inclusion
   \[ \tau^\ast\bigl( \psef^r(P)\bigr)\ \subset\ \psef^r(X)\ ;\]

\noindent
(\rom4) Either $r=2$, or the standard Lefschetz conjecture $B(X)$ holds.
% (i.e., numerical and homological equivalence coincide on $X$).

Then
    \[  \gr^0_F H^{j}(X)_{\text{van}}=0\ \ \ \hbox{for\ all\ }j\le n-r\ . \]
    \end{theorem}
    
  \begin{proof} First, in case $j<n-r$ we take generic hyperplane sections. That is, we consider (as before) a fibre diagram
    \[ \begin{array}[c]{ccc}
               X^\prime &\xrightarrow{\tau^\prime}& P^\prime\\
              \ \  \downarrow{{}^g}&& \downarrow\\
               X&\xrightarrow{\tau}&\ \ P\ ,\\
               \end{array}\]
               where $P^\prime\subset P$ is a smooth complete intersection of dimension $n^\prime+r$, and $X^\prime\subset X$ is smooth of dimension $n^\prime$, and we have equality
               $j=n^\prime-r$.
   
  \begin{lemma}\label{X} The class $\tau^\ast[X]\in N^r(X)$ is (homologically) big.
  \end{lemma}
  
 \begin{proof} Let $A\in N^1(P)$ be an ample divisor class. To prove bigness of $\tau^\ast[X]$, it suffices to prove
   \begin{equation}\label{incr}   \tau^\ast[X] -\epsilon \tau^\ast(A^r)\ \ \in \psef^r(X)\ ,\end{equation}
   for some $\epsilon >0$.
   
  Let $a\in\nef^{n-r}(X)$. It follows from assumption (\rom3) (by duality) that
     \[ \tau_\ast(a)\in\nef^n(P)\ .\]
     It follows from assumption (\rom1) that $\tau_\ast(a)$ is effective.
   Also, assumptions (\rom1) and (\rom2) combined imply that $[X]\in N^r(P)$ is big. Now, using the projection formula we find that
        \[    \bigl(\tau^\ast [X] -\epsilon \tau^\ast(A^r)\bigr) \cdot a= \bigl([X]-\epsilon A^r\bigr)\cdot  \tau_\ast (a)
                                                                                       \ge 0\ ,\]
                                                                                       for some sufficiently small $\epsilon >0$.
     This proves inclusion (\ref{incr}) and hence the bigness of $\tau^\ast[X]$.                                                                                       
   Since we have assumed that either $r=2$ or $B(X)$ holds, the two notions of bigness coincide (remark \ref{D}), and so $\tau^\ast[X]$ is homologically big.
    \end{proof}  
   
   \begin{lemma} The class $(\tau^\prime)^\ast[X^\prime]\in N^r(X^\prime)=N^r H^{2r}(X^\prime,\RR)$ is homologically big.
   \end{lemma}
   
   \begin{proof} The fact that $(\tau^\prime)^\ast[X^\prime]$ is big   
    can be deduced from lemma \ref{X} along the lines of the proof of lemma \ref{Xprime}.

    In case $r=2$, the two notions of bigness coincide (remark \ref{D}). Otherwise, since property $B(X)$ implies $B(X^\prime)$ \cite{K}, the two notions of bigness also coincide on $X^\prime$; this proves the lemma.
         \end{proof}
   
         Applying lemma \ref{lf} to the homologically big class $(\tau^\prime)^\ast[X^\prime]\ \in N^r(X^\prime)=N^r H^{2r}(X^\prime,\RR)$, we find that
   \[   \cup (\tau^\prime)^\ast[X^\prime]\colon\ \  \gr^0_F  H^{n^\prime-r}(X^\prime,\C)\ \to\ \gr^{r}_F H^{n^\prime+r}(X^\prime,\C)    \]
   is injective (and hence, for dimension reasons, an isomorphism).  On the other hand, it follows from the normal bundle formula that there is a factorization  
     \[\cup (\tau^\prime)^\ast[X^\prime]\colon\ \   \gr^0_F H^{n^\prime-r}(X^\prime,\C)\ \xrightarrow{(\tau^\prime)_\ast}\ \gr^r_F H^{n^\prime+r}(P^\prime,\C)\ 
      \xrightarrow{(\tau^\prime)^\ast}\ 
          \gr^r_F H^{n^\prime+r}(X^\prime,\C)  \ .\]
         We can thus conclude that
         \[ (\tau^\prime)_\ast\colon\ \ \gr^0_F H^{j}(X^\prime,\C)\ \to\ \gr^r_F H^{j+2r}(P^\prime,\C) \]
         is injective.
     
      To return to $X$, we consider a commutative diagram
         \[ \begin{array}[c]{ccc}
          \gr^0_F H^j(X,\C)&\xrightarrow{\tau_\ast}& \gr^r_F H^{j+2r}(P,\C)\\
          \downarrow&&\downarrow\\
         \gr^0_F H^j(X^\prime,\C)&\xrightarrow{(\tau^\prime)_\ast}& \gr^r_F H^{j+2r}(P^\prime,\C)\\
          \end{array}\]
          % (NB: this can be seen by using cohomology with supports, and forgetting the support !!!!)
          where vertical arrows are injective (this is an application of weak Lefschetz; note that $\dim X^\prime=n^\prime>j$ and $\dim P^\prime=j+2r$). It follows from this commutative diagram that
          \[ \tau_\ast\colon\ \ \gr^0_F H^j(X,\C)\ \to\ \gr^r_F H^{j+2r}(P,\C)\]
          is injective.
        
  \end{proof}

\begin{corollary}\label{AC+} Let $n,r$ be positive integers with $n\ge r+2$. Let $P$ be a Grassmannian of lines in a projective space, and $\dim P=n+r$. Let $X\subset P$ be a smooth bulky subvariety of dimension $n$. Assume either $r=2$ or $B(X)$ holds. Then
  \[ H^2(X,\ZZ)=\ZZ\ .\]
  \end{corollary}

\begin{proof} As mentioned in the introduction, Arrondo and Caravantes have proven \cite{AC} that $\pic(X)=\ZZ$. 

We now check that all assumptions of theorem \ref{main3} are satisfied. Any Grassmannian $P$ has $\nef^j(P)=\eff^j(P)$ for all $j$ so assumption (\rom1) is OK. Assumption (\rom2) is OK by remark \ref{compare}. Assumption (\rom3) of theorem \ref{main3} is satisfied, because (by homogeneity) any subvariety $a\subset P$ is homologically equivalent to a subvariety in general position with respect to $X$.
Applying theorem \ref{main3}, we find that $H^2(X,{\mathcal O}_X)=0$. The result now follows from the exponential sequence.
 \end{proof}
 
 \begin{corollary}\label{AC++} Let $P$ be a product $\PP^m\times\PP^m$, and let $X\subset P$ be a smooth subvariety of codimension $r$ and dimension $n\ge r+2$.
 Assume the two projection maps $X\to \PP^m$ are surjective. Assume also that either $r=2$ or $B(X)$ holds. Then
   \[ H^2(X,\ZZ)=\ZZ^2\ .\]
   \end{corollary}
   
   \begin{proof} Arrondo and Caravantes have proven that $\pic(X)=\ZZ^2$ \cite[Theorem 3.1]{AC}. The assumption about the projection maps ensures that $X$ is bulky \cite[Proposition 2.6]{Deb}, hence (by homogeneity of $P$) strictly nef. Applying theorem \ref{main3}, we find that $H^2(X,{\mathcal O}_X)=0$.
  \end{proof}
  
  \begin{definition}[Perrin \cite{Per2}] Let $\GG_Q(p,m)$ and $\GG_\omega(p,2m)$ be the Grassmannians of isotropic subspaces of dimension $p$ in a vector space of dimension $m$ (resp. $2m$) endowed with a non--degenerate quadratic form $Q$ (resp. symplectic form $\omega$).
  \end{definition}
  
  \begin{corollary}\label{NP++} Let $n,r$ be positive integers with $n\ge r+3$. Let $P$ be $\GG_Q(2,2m+1)$, $\GG_\omega(2,2m)$ or $\GG_Q(2,4m)$. Let $X\subset P$ be a smooth bulky subvariety of dimension $n$ and codimension $r$. Assume either $r=2$, or $B(X)$ holds. Then
    \[ H^2(X,\ZZ)=\ZZ\ .\]
    \end{corollary}
    
 \begin{proof} Perrin has proven that $\pic(X)=\ZZ$ \cite[Corollary 0.11]{Per2}. Since $P$ is homogeneous, the conditions of theorem \ref{main3} are again fulfilled, so we also have $H^2(X,{\mathcal O}_X)=0$.
    \end{proof}

\begin{remark} It would be interesting if one could prove theorem \ref{main3} (or even the corollaries \ref{AC+} and \ref{AC++} and \ref{NP++}) for $r>2$ without assuming some standard conjecture for the subvariety $X$.
I have not been able to do so.
\end{remark}

\vskip0.6cm

\begin{acknowledgements} This note is a belated fruit of the 2014 Marrakech workshop on cones of positive cycle classes, which was a great occasion to learn about the body of work \cite{FLe}, \cite{FLe2}, \cite{Le}. Thanks to all the participants of this workshop. Many thanks to Yasuyo, Kai and Len for coming to Marrakech with me. Thanks to the referee for several very helpful remarks.
\end{acknowledgements}


\vskip0.6cm
\begin{thebibliography}{dlPG99}



\bibitem{AC} E. Arrondo and J. Caravantes, On the Picard group of low--codimension subvarieties, Indiana Univ. Math. J. 58 no. 3 (2009), 1023---1050,

\bibitem{Art} M. Artin, On isolated rational singularities of surfaces, American Journal of Mathematics 88 (1) (1966), 129---136,

%\bibitem{BS} L. Barbieri--Viale and V. Srinivas, On the N\'eron--Severi group of a singular variety, J. reine u. angew. Math. 435 (1993), 65---82,

%\bibitem{BS2} L. Barbieri--Viale and V. Srinivas, The N\'eron--Severi group and the mixed Hodge structure on $H^2$. Appendix to: ``On the N\'eron--Severi group of a singular variety'',
%J. reine u. angew. Math. 450 (1994), 37---42,

\bibitem{Ba} W. Barth, Transplanting cohomology classes in complex-projective space, American Journal of Mathematics 92 (1970), 951---967,

\bibitem{BO} S. Bloch and A. Ogus, Gersten's conjecture and the homology of schemes, Ann. Sci. Ecole Norm. Sup. 4 (1974), 181---202,

\bibitem{Br} M. Brion, Vari\'et\'es sph\'eriques, available online at
{ http://www-fourier.ujf-grenoble.fr/\textasciitilde mbrion/spheriques.pdf},

\bibitem{Br2} M. Brion, Spherical varieties, available online at
{ http://www-fourier.ujf-grenoble.fr/\textasciitilde mbrion/notes$\_$bremen.pdf},

\bibitem{CCP} F. Campana, J. Chen and T. Peternell, Strictly nef divisors, Math. Ann. 342 (2008), 565---585,

\bibitem{CTS} J.--L. Colliot--Th\'el\`ene and A. Skorobogatov, Descente galoisienne sur le groupe de Brauer, J. reine u. angew. Math. 682 (2013), 141---165,

\bibitem{Deb} O. Debarre, Th\'eor\`emes de connexit\'e pour les produits d'espaces projectifs et les Grassmanniennes, American Journal of Mathematics 118 No. 6 (1996), 1347---1367,

\bibitem{DV} O. Debarre, L. Ein, R. Lazarsfeld and C. Voisin, Pseudoeffective and nef classes on
abelian varieties, Compos. Math. 147 no. 6 (2011), 1793---1818,

\bibitem{LFu} L. Fu, On the coniveau of certain sub--Hodge structures, Math. Res. Lett. 19 (2012), 1097---1116,

\bibitem{FLe} M. Fulger and B. Lehmann, Positive cones of dual cycle classes, arXiv:1408.5154v2, to appear in Alg. Geom.,

\bibitem{FLe2} M. Fulger and B. Lehmann, Zariski decompositions of numerical cycle classes, arXiv:1310.0538v3, to appear in J. of Alg. Geom.,

\bibitem{F} W. Fulton, Intersection theory, Springer Berlin Heidelberg New York 1984,

\bibitem{Fu} W. Fulton, On the topology of algebraic varieties. In: Algebraic Geometry, Bowdoin 1985, Proceedings of Symposia in Pure Mathematics, vol. 46. American Mathematical Society Providence 1987,

\bibitem{FL} W. Fulton and R. Lazarsfeld, Connectivity and its applications in algebraic geometry. In: Algebraic Geometry (Chicago, 1980), Lecture Notes in Mathematics 862, Springer Berlin 1981,

\bibitem{FMSS} W. Fulton, R. MacPherson, F. Sottile and B. Sturmfels, Intersection theory on spherical varieties, J. Alg. Geom 4 (1995), 181---193,

\bibitem{Har} R. Hartshorne, Varieties of small codimension in projective space. Bull. Am. Math. Soc. 80(6) (1974), 1017---1032,

\bibitem{K0} S. Kleiman, Algebraic cycles and the Weil conjectures, in: Dix expos\'es sur la cohomologie des sch\'emas, North--Holland Amsterdam, 1968, 359---386, 

\bibitem{K} S. Kleiman, The standard conjectures, in: Motives (U. Jannsen et alii, editors), Proceedings of Symposia in Pure Mathematics Vol. 55 (1994), Part 1, 

%\bibitem{moi} R. Laterveer, A remark on Barth's connectivity theorem, Manuscr. Math. 138 No. 1--2 (2012), 23---34,

\bibitem{Laz} R. Lazarsfeld, Positivity in Algebraic Geometry I, Positivity in Algebraic Geometry II, Springer Berlin Heidelberg New York 2004,

\bibitem{Le} B. Lehmann, Volume--type functions for numerical cycle classes, arXiv:1601.03276v1, to appear in Duke Math. J.,

\bibitem{QLi} Q. Li, Pseudo--effective and nef cones on spherical varieties, Math. Z. 280 (2015), 945---979,

\bibitem{Lie} D. Lieberman, Numerical and homological equivalence of algebraic cycles on Hodge manifolds, Amer. J. Math. 90 (1968), 380---405,

%\bibitem{Lyu} G. Lyubeznik, Etale cohomological dimension and the topology of algebraic varieties. Ann. Math. 137 (1993), 71---128,

%\bibitem{MNP} J. Murre, J. Nagel and C. Peters, Lectures on the theory of pure motives, Amer. Math. Soc. University Lecture Series 61, Providence 2013,

\bibitem{Per} N. Perrin, Small codimension smooth subvarieties in even-dimensional homogeneous spaces with Picard group $\ZZ$, 
Comptes Rendus de l'Acad\'emie des Sciences 345 no. 3 (2007), 155---160,

\bibitem{Per2} N. Perrin, Small codimension subvarieties in homogeneous spaces,
Indag. Math. (N.S.) 20 no. 4 (2009), 557---581,

\bibitem{Per3} N. Perrin, On the geometry of spherical varieties,
Transform. Groups. 19 no. 1 (2014), 171---223,

\bibitem{PS} C. Peters and J. Steenbrink, Mixed Hodge structures, Springer--Verlag, Berlin 2008,

%\bibitem{Sch} T. Scholl, Classical motives, in: Motives (U. Jannsen et alii, eds.), Proceedings of Symposia in Pure Mathematics Vol. 55 (1994), Part 1, 

\bibitem{Voi} C. Voisin, Coniveau $2$ complete intersections and effective cones, Geom. Funct. Anal. 19 (5) (2010), 1494---1513.


\end{thebibliography}
\end{document}